\documentclass[preprint,11pt]{elsarticle}
\usepackage[utf8]{inputenc}
\usepackage[T1]{fontenc}
\usepackage{amsmath,bm}
\usepackage{amsfonts}
\usepackage{amssymb,amsthm,multirow,booktabs}
\usepackage{makeidx,color}
\usepackage{graphicx,mathrsfs}
\usepackage[a4paper, total={7in, 9in}]{geometry}

\newtheorem{thm}{Theorem}[section]
\newtheorem{lem}[thm]{Lemma}

\ifpdf
\DeclareGraphicsExtensions{.eps,.pdf,.png,.jpg}
\else
\DeclareGraphicsExtensions{.eps}
\fi

\usepackage{amssymb}
\usepackage{amsmath}

\journal{Applied Mathematics Letters}

\begin{document}

\begin{frontmatter}



\title{Strong Stability Preserving Integrating Factor Runge-Kutta Methods for Differential Lyapunov Equations with Positivity Preservation} 


		\author[ouc]{Wenhui Zhang}
        \ead{zwh3513@stu.ouc.edu.cn}
\author[ouc,lmm]{Rihui Lan\corref{cor}}
\ead{lanrihui@ouc.edu.cn}
\author[seu]{Xiaobo Jing}
\ead{xiaobo@seu.edu.cn}
\cortext[cor]{Corresponding author}

\affiliation[ouc]{organization={School of Mathematical Sciences, Ocean University of China},
            city={Qingdao, Shandong},
            postcode={266100}, 
            country={China}}
\affiliation[lmm]{organization={Laboratory of Marine Mathematics, Ocean University of China},
	city={Qingdao, Shandong},
	postcode={266100}, 
	country={China}}
\affiliation[seu]{organization={School of Mathematics, Southeast University},
	city={Nanjing, Jiangsu},
	postcode={210096}, 
	country={China}}
\begin{abstract}
This paper introduces second- and third-order integrating factor strong stability preserving Runge-Kutta   methods for solving differential Lyapunov equations. The proposed schemes break 
the traditional order barrier while rigorously preserving the symmetry and positive semidefiniteness (SPSD) of numerical solutions—an essential property for stability analysis and control-theoretic applications. Furthermore, we provide a rigorous error estimate for the second-order scheme. Numerical experiments validate the accuracy of the methods and their ability to maintain SPSD properties.
\end{abstract}




\end{frontmatter}


\vspace{-3.5ex}
	\section{Introduction}
    \vspace{-1.5ex}
In this paper, we consider the following differential Lyapunov equation (DLE): 
\begin{equation}\label{dle:ori}
	\begin{aligned}
		&\dot{X}(t)=A X(t)+X(t) A^\top+Q(t), \quad X(t_0)=X_0,
	\end{aligned}
\end{equation}
where  $X(t)\in \mathbb R^{n\times n}$ is the unknown, $X_0\in \mathbb R^{n\times n}$ is the initial value satisfying symmetric postive semidefiniteness.  The coefficient matrix $A\in \mathbb R^{n\times n}$. $Q(t)\in \mathbb R^{n\times n}$ is symmetric positive semidefinite (SPSD), which can depend on $X(t)$.  For convinience, we write $Q\geq 0$ if $Q$ is SPSD. 

DLEs play a pivotal role in control theory \cite{abou2003matrix}, dynamical systems \cite{schmidt2023rank}, and numerical methods for partial differential equations (PDEs) \cite{massei2018solving}.
In control theory, DLEs are fundamental for stability analysis and robustness verification of linear and nonlinear systems. For example, \cite{schmidt2023rank} employed the Kalman filter for inference in high-dimensional dynamical systems, where the process-noise covariance matrix is derived from solutions of DLEs, ensuring accurate state estimation.
Beyond control, DLEs also arise in PDE discretizations, particularly when reformulating problems in matrix form rather than vector form. Such approach can significantly reduce the dimensionality of the coefficient matrix, which is particularly advantageous when combined with exponential integrator methods, offering computational efficiency for large-scale systems.

%
Recently, the numerical solutions of DLEs gain much attention. For stationary DLEs, the review paper \cite{simoncini2016computational} provides several numerical methods. For non-stationary DLEs, \cite{li2021exponential} considered some Rosenbrock type exponential integrators methods for Lyapunov and Ricatti differential equations. However, they did not take SPSD preservation (also referred to as positivity preservation) into their consideration. In \cite{gillis2013positive}, Gillis and Diehl proposed PDPLD to preserve the positiveness of Lyapunov equation. However, PDPLD is an implicit nonlinear scheme in the framework of optimization. 
It is challenging to design the high-order SPSD preservation schemes, because the direct methods bear an order barrier as stated in \cite[Theorem 2.3]{dieci1994positive}. It is stated that only the first-order direct scheme is SPSD for Lyapunov and Riccati differential equations,  which suffers from the effect of general matrix $A$. Alternatively, \cite{dieci1994positive} proposed Hamiltonian system with symplectic Runge-Kutta method and the linearized mid-point scheme to achieve higher order scheme. However, the first approach is computationally cost, and the second one is limited to second-order accuracy. 

This paper presents an investigation into high-order SPSD-preserving time integration schemes for DLEs. Time marching is achieved via the integrating factor strong stability preserving Runge-Kutta (IF-SSPRK) methods  \cite{ahmed2019high}, which  leverages the strong stability-preserving (SSP) property \cite{gottlieb1998total}. The IF-SSPRK methods represent an advancement over conventional Runge-Kutta (RK) techniques, specifically tailored for systems exhibiting linear stiffness. Its fundamental strategy employs an exponential integrating factor to absorb the stiff linear term, permitting the application of standard RK methods to the resultant system. 

The rest of paper is structured as follows. Section \ref{sec:scheme} presents the second and third order IF-SSPRK schemes. Their numerical analysis is conducted in Section \ref{sec:numa} and the numerical experiments are performed in Section \ref{sec:numexp}. Lastly, some conclusion is addressed in Section \ref{sec:con}.


\section{Integrating-factor strong stability-preserving Runge Kutta methods for DLE}\label{sec:scheme}	
Given by the generalized variation-of-constants formula, we can have the exact formula for DLEs, which is stated in Lemma  \ref{sol:for}.
\begin{lem}\label{sol:for}\cite{abou2003matrix}
	The unique solution of the general Lyapunov differential equation
	$$
	\dot{X}(t)=A(t) X+X A(t)^\top+Q(t), \quad X\left(t_0\right)=X_0
	$$
	is defined by
	$$
	X(t)\textstyle =\Phi_A\left(t, t_0\right) X_0 \Phi_A^\top\left(t, t_0\right)+\int_{t_0}^t \Phi_A(t, s) Q(s) \Phi_A^\top(t, s) d s
	$$
	where the transition matrix $\Phi_A\left(t, t_0\right)$ is the unique solution to the problem
	$$
	\dot{\Phi}_A\left(t, t_0\right)=A(t) \Phi_A\left(t, t_0\right), \quad \Phi_A\left(t_0, t_0\right)=I .
	$$
	Futhermore, if $A$ is assumed to be a constant matrix, then we have
	\begin{equation}\label{sol}
		X(t) \textstyle =e^{\left(t-t_0\right) A} X_0 e^{\left(t-t_0\right) A^\top}+\int_{t_0}^t e^{(t-s) A} Q(s) e^{(t-s) A^\top} d s.
	\end{equation}
\end{lem}
Thanks to Lemma \ref{sol:for}, we can first multiply Eq. \eqref{sol}  by $e^{-t A}$ on the left and by $e^{-t A^\top}$ on the right, then take the derivative on $t$ to have
\begin{equation}\label{ifrk:ode}
	\textstyle 	\frac{d e^{-t A} X(t)  e^{-t A^\top}}{dt} = e^{-t A} Q(t) e^{-t A^\top}.
\end{equation}
Let $W(t) = e^{-t A} X(t)  e^{-t A^\top}$,  we have
\begin{equation}\label{ode:if}
	\begin{aligned}
		\textstyle \frac{dW}{dt} & \textstyle = e^{-t A}  Q(t) e^{-t A^\top}:=N(t).
	\end{aligned}
\end{equation}
Based	on Eq. \eqref{ode:if}, we can propose the second and third order strong stability-preserving Runge-Kutta (SSPRK) methods \cite{isherwood2018strong}. SSPRK schemes is constructed by the convex combination of  different time steps. Hence, they can gain the better stability.   
\paragraph{IF-SSPRK2}  Let $\varphi^n$ be the value at $t=t_n$. Then, the second-order integrating factor strong stability-preserving Runge–Kutta (IF-SSPRK2) scheme for solving \eqref{ifrk:ode}  is given as 
\begin{equation}\label{ifrk2}
	\left\{
	\begin{aligned}
		\widehat{W}^{n+1} & = W^n + \tau N^n,\\
		\widehat{W}^{n+2}& = \widehat{W}^{n+1} + \tau N^{n+1},\\
		W^{n+1} & \textstyle= \frac12 W^n + \frac12 \widehat{W}^{n+2},
	\end{aligned}
	\right.  
	\Rightarrow
	\left\{
	\begin{aligned}
		\widehat{X}^{n+1} &= e^{\tau A}X^n e^{\tau A^\top} + \tau  e^{\tau A}Q^n e^{\tau A^\top},\\
		\widehat{X}^{n+2}&= e^{\tau A}\widehat{X}^{n+1} e^{\tau A^\top} + \tau  e^{\tau A}Q^{n+1} e^{\tau A^\top},\\
		X^{n+1}& \textstyle= \frac12 e^{\tau A}X^{n} e^{\tau A^\top} +  \frac12 e^{-\tau A}\widehat{X}^{n+2} e^{-\tau A^\top}.
	\end{aligned}
	\right.
\end{equation}

\paragraph{IF-SSPRK3} The third-order integrating factor strong stability-preserving Runge–Kutta (IF-SSPRK3) scheme for solving \eqref{ifrk:ode} is given as 
\begin{equation}\label{ifrk3}
	\left\{
	\begin{aligned}
		\widehat{W}^{n+1} & = W^n + \tau N^n,\\
		\widehat{W}^{n+2}& = \widehat{W}^{n+1} + \tau N^{n+1},\\
		\widehat{W}^{n+\frac12} & \textstyle= \frac34 W^n + \frac14 \widehat{W}^{n+2},\\
		\widehat{W}^{n+\frac32} & = \widehat{W}^{n+\frac12} + \tau N^{n+\frac12},\\
		W^{n+1} & \textstyle= \frac13 W^n + \frac23\widehat{W}^{n+\frac32},
	\end{aligned}
	\right.  
	\Rightarrow
	\left\{
	\begin{aligned}
		\widehat{X}^{n+1} &= e^{\tau A}X^n e^{\tau A^\top} + \tau  e^{\tau A}Q^n e^{\tau A^\top},\\
		\widehat{X}^{n+2}&= e^{\tau A}\widehat{X}^{n+1} e^{\tau A^\top} + \tau  e^{\tau A}Q^{n+1} e^{\tau A^\top},\\
		\widehat{X}^{n+\frac12}&\textstyle=\frac34  e^{\frac12\tau A}X^{n} e^{\frac12\tau A^\top} +\frac14  e^{-\frac32\tau A}\widehat{X}^{n+2} e^{-\frac32\tau A^\top} ,\\
		\widehat{X}^{n+\frac32}&= e^{\tau A}\widehat{X}^{n+\frac12} e^{\tau A^\top} + \tau  e^{\tau A}Q^{n+\frac12} e^{\tau A^\top},\\
		X^{n+1}& \textstyle= \frac13  e^{\tau A}X^{n} e^{\tau A^\top} + \frac23e^{-\frac12\tau A}\widehat{X}^{n+\frac32} e^{-\frac12\tau A^\top}.
	\end{aligned}
	\right.
\end{equation}
\vspace{-4.5ex}
\section{SPSD preservation and error estimate}\label{sec:numa}
\vspace{-1.5ex}
For simplicity, we will aim at IF-SSPRK2 scheme. The analysis can be easily extended to IF-SSPRK3. First of all, we will show that IF-SSPRK2 scheme preserves the symmetry and positive definiteness.  As stated in \cite[Theorem 2.3]{dieci1994positive}, matrix $A$ will violate the SPSD preservation for the direct higher order scheme. 
However, IF-SSPRKs have $e^{-t A}$ and $e^{-t A^\top} $ lying in the two sides of intermediate quantities, so that $A$ will not ruin SPSD preservation. 
\begin{thm}
	Assume that $X^0$ is SPSD. Then, for any time step $\tau>0$ and $n\geq 0$, 
	the solution $X^{n+1}$ to \eqref{ifrk2} is SPSD.
\end{thm}
\begin{proof}
	Given $X^0$ and $Q(t)$ are SPSD, it is easy to see that $X^{n+1}$ is SPSD by the mathematical induction and convex combination between two different stage values.
\end{proof}
Next, we will show that IF-SSPRK2 scheme is indeed second-order accurate. For simplicit, we will just consider the case $Q(t)$ is independent of $X(t)$. Let  $\Vert \cdot \Vert$ be any vector norm and the induced matrix norm is defined as $\textstyle \Vert B \Vert = \max\limits_{\Vert \bm v\Vert=1}\Vert B\bm v\Vert$. 
In addition, the logarithmic norm $\mu(B)$ of $B\in \mathbb{C}^{n\times n}$ is defined as $\mu(B):=\{\max \operatorname{Re} z: z \in \mathscr{F}(B)\}$, where $\mathscr{F}(B)=\left\{x^* B x: x \in \mathbb{C}^d,\|x\|=1\right\}$.
Assume there exists a positive constant $\mathcal{C}>0$ independent of $\tau$, such that $\mu(A) \leq \mathcal{C}$. Then, we have the following estimate on $e^{sA}$.
\begin{lem}\cite{abou2003matrix}
	For any matrix $A\in \mathbb{R}^{n\times n}$ and any constant $s\geq0$, there holds
	$$\Vert e^{s A} \Vert \leq e^{s \mu(A)} \leq e^{s \mathcal{C}}.$$
\end{lem}
\begin{thm}
	Assume $\tau\leq 1$ and $\textstyle \max\limits_{t\in [0, T]} \max \left\{ \Vert Q(t)\Vert, \left\Vert \frac{\partial Q(t)}{\partial t}\right\Vert, \left\Vert  \frac{\partial^2 Q(t)}{\partial t^2} \right\Vert \right\}\leq \mathcal{M}$. Given $X(t_{n+1})$ is the solution to DLE \eqref{dle:ori}. The solution $X^{n+1}$ to \eqref{ifrk2} has the following estimate,
	\begin{equation}\label{ifrk2:es}
		\Vert X(t_{n+1}) - X^{n+1}\Vert \leq \exp(2\mathcal{C}e^{2\mathcal{C}}t_{n+1})  \frac{2\mu^2 + 6\mu + 1}{2}e^{2\mu}\mathcal{M} t_{n+1}\tau^2.
	\end{equation}
\end{thm}
\begin{proof}
	First of all, we have
	\begin{equation}\label{reform}
		\left\{
		\begin{aligned}
			X(t_{n+1}) &= e^{\tau A}X(t_n) e^{\tau A^\top} + \tau  e^{\tau A}Q^n e^{\tau A^\top} + \mathcal{R}_1,\\
			X(t_{n+2})  &= e^{\tau A}X(t_{n+1})  e^{\tau A^\top} + \tau  e^{\tau A}Q^{n+1} e^{\tau A^\top} + \mathcal{R}_2,\\
			X(t_{n+1})  & \textstyle= \frac12 e^{\tau A} X(t_{n})  e^{\tau A^\top} +  \frac12 e^{-\tau A} X(t_{n+2})  e^{-\tau A^\top} + \mathcal{R}_3,
		\end{aligned}
		\right.
	\end{equation}
	where
	\begin{equation*}
		\begin{aligned}
			\mathcal{R}_1 & \textstyle = \int_{t_n}^{t_{n+1}} \left( e^{(t_{n+1}-s) A} Q(s) e^{(t_{n+1}-s) A^\top}  - e^{\tau A}Q(t_n) e^{\tau A^\top}\right)d s,\\
			\mathcal{R}_2 & \textstyle = \int_{t_{n+1}}^{t_{n+2}} \left( e^{(t_{n+2}-s) A} Q(s) e^{(t_{n+2}-s) A^\top}  - e^{\tau A}Q(t_{n+1}) e^{\tau A^\top}\right)d s,
		\end{aligned}
	\end{equation*}
	and
	\begin{equation*}
		\begin{aligned}
			\mathcal{R}_3 =&  \textstyle \frac12 \int_{t_n}^{t_{n+1}}  e^{(t_{n+1}-s) A} Q(s) e^{(t_{n+1}-s) A^\top}   d s  -  \textstyle \frac12 \int_{t_{n+1}}^{t_{n+2}}  e^{(t_{n+1}-s) A} Q(s) e^{(t_{n+1}-s) A^\top}  d s.
		\end{aligned}
	\end{equation*}
	Define the errors $E^{n+1} = X(t_{n+1}) - X^{n+1}$ and $\widehat{E}^{n+1} = X(t_{n+1}) - \widehat{X}^{n+1}$, then we have the following error equations by substracting \eqref{ifrk2}  from \eqref{reform},
	\begin{equation*}
		\left\{
		\begin{aligned}
			\widehat{E}^{n+1} &= e^{\tau A} E^n e^{\tau A^\top} + \mathcal{R}_1,\\
			\widehat{E}^{n+2} &= e^{\tau A}	\widehat{E}^{n+1}  e^{\tau A^\top} + \mathcal{R}_2,\\
			E^{n+1} & \textstyle= \frac12 e^{\tau A} E^{n}   e^{\tau A^\top} +  \frac12 e^{-\tau A} \widehat{E}^{n+2}  e^{-\tau A^\top} + \mathcal{R}_3.
		\end{aligned}
		\right.
	\end{equation*}
	The error equations can be further reduced as one single equation, that is 
	\begin{equation*}
		\begin{aligned}
			\textstyle	E^{n+1} & \textstyle = ~ e^{\tau A} E^n e^{\tau A^\top} + \frac12 \mathcal{R}_1 + \frac12 e^{-\tau A} \mathcal{R}_2  e^{-\tau A^\top} + \mathcal{R}_3\\
			& \textstyle = ~ e^{\tau A} E^n e^{\tau A^\top} + \int_{t_n}^{t_{n+1}}  e^{(t_{n+1}-s) A} Q(s) e^{(t_{n+1}-s) A^\top}   d s - \frac12 \tau e^{\tau A} Q(t_n) e^{\tau A^\top} - \frac12 Q(t_{n+1}).
		\end{aligned}
	\end{equation*}
	By $e^{2\tau\mathcal{C}}\leq 1+2\mathcal{C}e^{2\mathcal{C}} \tau$, the trapezoidal rule, and the triangle inequality, there holds
	\begin{equation*}
		\Vert E^{n+1} \Vert \leq (1+2\mathcal{C}e^{2\mathcal{C}} \tau) \Vert E^n \Vert + \frac{2\mu^2 + 6\mu + 1}{2}e^{2\mu\tau}\mathcal{M}\tau^3.
	\end{equation*}
	Changing index $n$ to $i$ and summing over from $i=0$ to $i=n$, it yields
	\begin{equation*}
		\Vert E^{n+1} \Vert\leq  2\mathcal{C}e^{2\mathcal{C}}  \sum\limits_{i=1}^n \tau \Vert E^i \Vert + \frac{2\mu^2 + 6\mu + 1}{2}e^{2\mu}\mathcal{M} t_{n+1}\tau^2,
	\end{equation*}
	where  $\Vert E^0 \Vert=0$ is applied. Lastly, by Gr\"onwall's inequality, we have the estimate \eqref{ifrk2:es}.
\end{proof}
\section{Numerical Experiments}\label{sec:numexp}
In this section, we will consider two numerical experiments: one is for the Allen-Cahn (AC) equations and its modification; the other one comes from a linear quadratic regulator (LQR) problem for a convection-diffusion system.
The matrix exponential function is computed by Matlab built-in function $expm$.
\vspace{-2ex}
\subsection{Convergent tests}
AC equation is given as:
$
	\frac{\partial u}{\partial t} = \varepsilon^2\Delta u + u - u^3, \text{ in } \Omega\times(0,1],
$
where $\Omega = (-1,1)\times(-1,1)$, $\varepsilon=0.01$ is the interface width parameter. We consider the homogeneous Dirichlet boundary condition here. The spatial discretization is by the central difference method, where $h_x=h_y=0.01$ and $A$ is the resulting discrete matrix relating to $\varepsilon^2\partial_{xx}$.  $U=(u_{i,j})_{i,j=1}^{200}$ is the discrete solution, where $u_{i,j}$ corresponds to the solution at $(x_i, y_j)$. Then, the spatial semi-discretization of AC equation is 
\begin{equation}\label{hadam}
	\frac{\partial U}{\partial t} = A U +  UA^\top + U - U^{\circ 3},
\end{equation}
where $U^{\circ 3}:=(u_{i,j}^3)_{i,j=1}^{200}$ is the Hadamard operator. The initial value is chosen as follows: we first randomly generated the orthogonal matrix $P$ and a random diagonal matrix $S$, where first diagonal element of $S$ is 0 and other diagonal elements are between 0 and 1. Eventually, the initial value is picked by $U^0=PSP^\top$. We compute the solution of $\tau=2^{-10}$ generated by IF-SSPRK3 as our reference solution. The errors produced by IF-SSPRK2 and IF-SSPRK3 are present in Table \ref{convergence}, where we can observe the exact second- and third-order accuracy for IF-SSPRK2 and IF-SSPRK3, respectively.
\begin{table}[!bht]
	\centering\footnotesize
	\begin{tabular}{c|cc|cc}
		\hline
		\multirow{2}{*}{$\tau$}& \multicolumn{2}{c|}{IF-SSPRK2}& \multicolumn{2}{c}{IF-SSPRK3}\\\cmidrule(lr){2-3}\cmidrule(lr){4-5}
		&Error &Rate &Error &Rate \\ \hline
		$2^{-3}$  &6.98E-3&-     &3.68E-4&-\\
		$2^{-4}$  &1.80E-3&1.96&4.61E-5&3.00\\
		$2^{-5}$  &4.52E-4&1.99&5.71E-6&3.01\\
		$2^{-6}$  &1.13E-4&2.00 &7.09E-7&3.01\\
		\hline
	\end{tabular}
    \vspace{-0.5ex}
	\caption{Convergence rates for AC equation produced by the  IF-SSPRK2 and IF-SSPRK3 schemes.}\label{convergence}
\end{table}

Since $U-U^{\circ 3}$ is generally not SPSD, we modify the Hadamard operator in \eqref{hadam} by the matrix powers to validate the SPSD preservation. That is
\begin{equation}\label{pow}
	\frac{\partial U}{\partial t} = A U +  UA^\top + U - U^3,
\end{equation}
where $U-U^3\geq 0$ when the eigenvalues of $U$ lie in $[0,1)$. Due to the round-off error in floating-point arithmetic across computational stages, we choose $U^0+\text{3E-15}\Vert U^0\Vert_FI$ as our initial value, where $U^0$ is generated by the same way mentioned as above. For comparison, we also test the problem with the regular SSPRK2 and SSPRK3. Choosing $T=1$ and $\tau=2^{-7}$, the evolution of the smallest eigenvalue produced by these four numerical schemes are present in Figure \ref{fig:eig}. It is observed both IF-SSPRK2 and IF-SSPRK3 can keep the smallest eigenvalues non-negative, however the regular SSPRK2 and SSPRK3 have the negative eigenvalues. 

      \begin{figure}[htbp]
	\centering
	\includegraphics[width=0.3\textwidth]{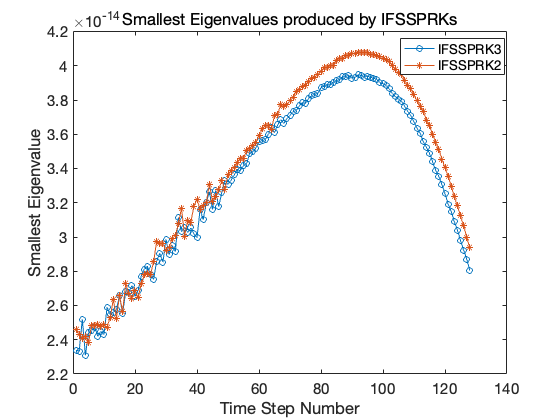}
	\includegraphics[width=0.3\textwidth]{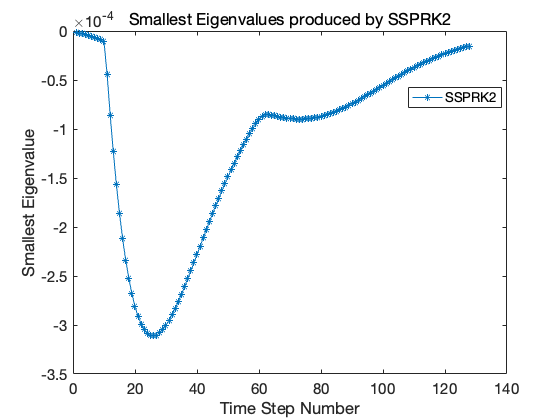}
	\includegraphics[width=0.3\textwidth]{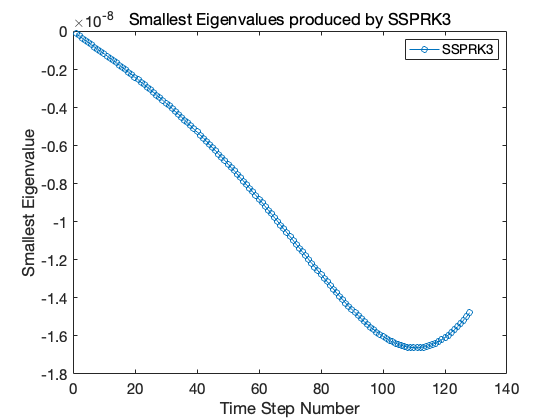}
    \vspace{-0.5ex}
	\caption{The evolution of smallest eigenvalue produced by IF-SSPRK (left), SSPRK2 (mid), and SSPRK3 (right).}
	\label{fig:eig}
\end{figure}
\vspace{-3.5ex}
\subsection{SPSD preservation with nonsymmetric matrix $A$}
In this subsection, we consider one DLE arise from a  LQR problem for a convection-diffusion system \cite{chen2025low}, where the matrix $A$ is obtained by the centered difference scheme on the equation $\frac{\partial w}{\partial t}=\Delta w-10x \frac{\partial w}{\partial x}-y \frac{\partial w}{\partial y}$, which is defined on  $\Omega= (0,1)^2$ with homogenous Dirichlet boundary conditions. 
In addition, the matrix $Q=\bm C \bm C^\top + \text{3E-15}\Vert \bm C \bm C^\top \Vert_F I$ where $\bm C\in \mathbb R^{d\times 1}$ with  $\bm C_j = 1$ if $x_j\in (0.7,0.9]$, otherwise $\bm C_j=0$. The time step $\tau=2^{-7}$ and the spatial mesh size $h_x=h_y=\frac{1}{21}$, so $A\in \mathbb R^{200\times 200}$. The smallest eigenvalues generated by IF-SSPRKs and regular SSPRKs are present in Figure \ref{fig:eig:2}. As we can see that, IF-SSPRKs can preserve SPSD pretty well. However, SSPRK schemes have the negative eigenvalues show-up.
      \begin{figure}[htbp]
	\centering
	\includegraphics[width=0.3\textwidth]{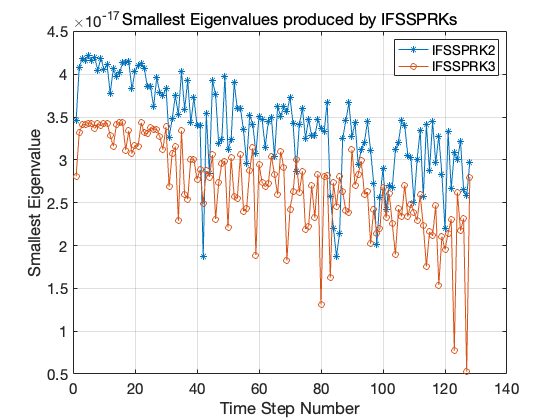}
	\includegraphics[width=0.3\textwidth]{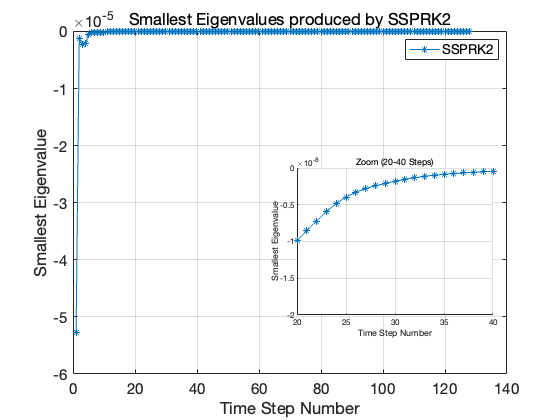}
	\includegraphics[width=0.3\textwidth]{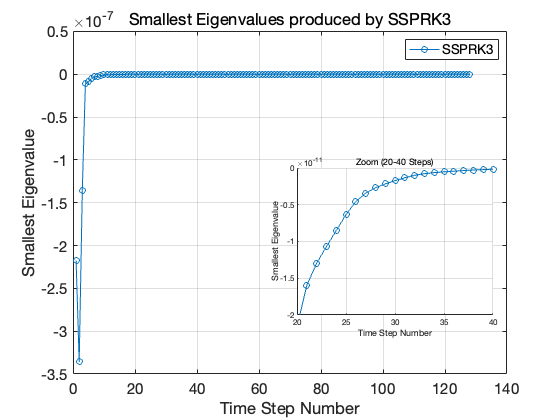}
    \vspace{-0.5ex}
	\caption{The evolution of smallest eigenvalues produced by IF-SSPRKs (left), SSPRK2 (mid), and SSPRK3 (right).}
	\label{fig:eig:2}
\end{figure}
\section{Conclusions}
\label{sec:con}
In this paper, we proposed the higher order SPSD-preserving numerical schemes for DLEs via the integrator factor strong stability preserving Runge-Kutta methods. We also provided the error estimates on the proposed schemes. Since the large scale scenario often occurs in the real engineering, we would leave the low-rank approximation as our future study.

\bibliographystyle{abbrv}
\bibliography{paraexpDRE}
\end{document}